\documentclass
{amsart}
\usepackage{graphicx}
\newtheorem{thm}{Theorem}[section]
\theoremstyle{definition}
\newtheorem{cor}[thm]{Corollary}
\newtheorem{prop}[thm]{Proposition}
\newtheorem{defn}[thm]{Definition}
\newtheorem{lem}[thm]{Lemma}

\newtheorem{rem}[thm]{Remark}
\newtheorem{ex}[thm]{Example}

\numberwithin{equation}{section}
\begin{document}
{\footnotetext{This research was supported with a grant from Farhangian University}
\title[$S$-secondary submodules of a module]{$S$-secondary submodules of a module}
\author{Faranak
 Farshadifar}
\address{Department of Mathematics, Farhangian University, Tehran, Iran.}
\email{f.farshadifar@cfu.ac.ir}

\begin{abstract}
Let $R$ be a commutative ring with identity, $S$ a multiplicatively closed subset of $R$, and $M$ be an $R$-module.
The aim of this paper is to introduce and investigate some properties of the notion of $S$-secondary submodules of $M$ as a generalization of secondary and $S$-second submodules of $M$. 

\end{abstract}


\maketitle
\section{Introduction}
\noindent
Throughout this paper, $R$ will denote a commutative ring with
identity and $\Bbb Z$ will denote the ring of integers.

Consider a non-empty subset $S$ of $R$. We call $S$ a multiplicatively closed subset of $R$ if (i) $0 \not \in S$, (ii) $1 \in S$, and (iii) $s\acute{s} \in S$ for all $s, \acute{s} \in S$ \cite{WK16}. Note that $S = R-P$ is a  multiplicatively closed subset of $R$ for every prime ideal $P$ of $R$.
Let $M$ be an $R$-module. A proper submodule $P$ of $M$ is said to be \emph{prime} if for any $r \in R$ and $m \in M$ with $rm \in P$, we have $m \in P$ or $r \in (P:_RM)$ \cite{Da78}. A non-zero submodule $N$ of $M$ is said to be \emph{second} if for each $a \in R$, the homomorphism $ N \stackrel {a} \rightarrow N$ is either surjective
or zero \cite{Y01}.

Let $S$ be a  multiplicatively closed subset of $R$ and let $P$ be a submodule of an $R$-module $M$ with $(P :_R M) \cap S =\emptyset$. Then the submodule
$P$ is said to be an \emph{$S$-prime submodule} of $M$ if there exists a fixed $s\in S$, and whenever $am \in P$, then $sa \in (P :_R M)$ or
$sm \in P$ for each $a \in R$, $m \in M$ \cite{satk19}. Particularly, an ideal $I$ of $R$ is said to be an \emph{$S$-prime ideal} if $I$ is an $S$-prime submodule of the $R$-module $R$.

In \cite{FF22} F. Farshadifar,  introduced  the notion of $S$-second submodules as a dual notion of $S$-prime submodules and investigated some properties of this class of modules.
Let $S$ be a  multiplicatively closed subset of $R$ and $N$ a submodule of an $R$-module $M$ with $Ann_R(N) \cap S =\emptyset$. Then the submodule
$N$ is said to be an \emph{$S$-second submodule} of $M$ if there exists a fixed $s\in S$, and whenever $rN\subseteq K$, where $r \in R$ and $K$ is a submodule of $M$, implies either that $rsN=0$ or $sN\subseteq K$ \cite{FF22}.

Let $S$ be a  multiplicatively closed subset of $R$ and $M$ be an $R$-module.
The main purpose of this paper is to introduce the notion of $S$-secondary submodules of an $R$-module $M$ as a generalization of secondary submodules of $M$ and
provide some information about this class of submodules. Also, this notion can be regarded as a generalization of  $S$-second submodules.

\section{Main results}
\noindent
Let $M$ be an $R$-module. A proper submodule $N$ of
$M$ is said to be \emph{completely irreducible} if $N=\bigcap _
{i \in I}N_i$, where $ \{ N_i \}_{i \in I}$ is a family of
submodules of $M$, implies that $N=N_i$ for some $i \in I$. It is
easy to see that every submodule of $M$ is an intersection of
completely irreducible submodules of $M$ \cite{FHo06}.

\begin{rem}\label{r2.1}
Let $N$ and $K$ be two submodules of an $R$-module $M$. To prove $N\subseteq K$, it is enough to show that if $L$ is a completely irreducible submodule of $M$ such that $K\subseteq L$, then $N\subseteq L$  \cite{AF101}.
\end{rem}

\begin{thm}\label{t2.3} Let $S$ be a  multiplicatively closed subset of $R$.  For a submodule $N$ of an $R$-module $M$  with $\sqrt{Ann_R(N)} \cap S= \emptyset$ the following statements are equivalent:
\begin{itemize}
 \item [(a)]  There exists a fixed $s \in S$ such that for each $r \in R$, $srN=sN$ or  $(sr)^tN=0$ for some $t \in \Bbb N$;
 \item [(b)]  There exists a fixed $s \in S$ and whenever $rN\subseteq K$, where $r \in R$ and $K$ is a submodule of $M$, implies either that $(rs)^tN=0$  for some $t \in \Bbb N$ or $sN\subseteq K$;
 \item [(c)] There exists a fixed $s \in S$ and whenever $rN\subseteq L$, where $r \in R$ and $L$ is a
completely irreducible submodule of $M$, implies either that $(rs)^tN=0$  for some $t \in \Bbb N$
 or $sN\subseteq L$.
 \item [(d)] There exists a fixed $s \in S$, and $JN\subseteq K$ implies $sJ \subseteq \sqrt{Ann_R(N)}$ or $sN \subseteq K$ for each ideal $J$ of $R$ and submodule $K$ of $M$.
\end{itemize}
\end{thm}
\begin{proof}
$(a) \Rightarrow (b)$ and $(b) \Rightarrow (c)$ are clear.

$(c) \Rightarrow (a)$
By part (c), there exists a fixed $s \in S$. Assume that $(sr)^tN\not=0$ for some  $r \in R$ and each $t \in \Bbb N$.
If $rsN \subseteq L$ for some completely irreducible submodule $L$ of $M$, then by assumption, $sN\subseteq
L$. Hence, by Remark \ref{r2.1}, $sN\subseteq rsN$, as required.

$(b)\Rightarrow (d)$
Suppose that  $JN\subseteq K$ for some ideal $J$ of $R$ and submodule $K$ of
$M$. By part (b), there is an $s \in S$ so that $rN\subseteq K$ implies $sr \in \sqrt{Ann_R(N)}$ or  $sN \subseteq K$ for each $r \in R$. Assume that $sN \not \subseteq K$. Then by Remark \ref{r2.1}, there exists a completely irreducible submodule $L$ of $M$ such that $K \subseteq L$ but $sN \not \subseteq L$. Then note that for each $a \in J$, we
have $aN \subseteq L$. By part (b), we can conclude that $sa \in \sqrt{Ann_R(N)}$ and so $sJ \subseteq \sqrt{Ann_R(N)}$.

$(d)\Rightarrow (b)$
Take $a \in R$ and $K$ a submodule of $M$ with $aN \subseteq K$. Now, put $J = Ra$. Then we have $JN \subseteq K$. By assumption, there is a fixed $s \in S$ such that $sJ = Ras \subseteq \sqrt{Ann_R(N)}$ or
$sN \subseteq K$ and so either $sa \in \sqrt{Ann_R(N)}$ or $sN \subseteq K$ as needed.
\end{proof}

\begin{defn}\label{d2.1}
Let $S$ be a  multiplicatively closed subset of $R$ and $N$ be a submodule of an $R$-module $M$ such that $\sqrt{Ann_R(N)} \cap S= \emptyset$.
We say that $N$ is an \textit{$S$-secondary submodule of $M$} if satisfies the equivalent conditions
of Theorem \ref{t2.3}. By an \textit{$S$-secondary module}, we mean
a module which is an $S$-secondary submodule of itself.
\end{defn}

\begin{rem}\label{r2.1}
Let $M$ be an $R$-module and $S$ a  multiplicatively closed subset of $R$. Clearly, every $S$-second submodule of $M$ is an $S$-secondary submodule of $M$. But the converse is not true in general. For example, Consider $\Bbb Z_4$ as an $\Bbb Z$-module. Set $S:=\Bbb Z \setminus 2\Bbb Z$. Then for each $s \in S$, $2\Bbb Z_4=2s\Bbb Z_4 \not=s\Bbb Z_4=\Bbb Z_4$ and $2s\Bbb Z_4 \not=0$ implies that  $\Bbb Z_4$ is not an  $S$-second $\Bbb Z$-module. But  $\Bbb Z_4$ is an  $S$-secondary $\Bbb Z$-module.  Because, if we consider $s=1$, and $n \in \Bbb Z$, then  $\Bbb Z_4=n(1) \Bbb Z_4=(1)\Bbb Z_4=\Bbb Z_4$, where $n \not =2k$ for each $k \in \Bbb N$ and $(n(1))^2\Bbb Z_4=0$, where $n =2k$ for some $k \in \Bbb N$.
\end{rem}

The following lemma is known, but we write it here for the sake of reference.
\begin{lem}\label{l2.1}
Let $M$ be an $R$-module, $S$ a  multiplicatively closed subset of $R$,
and $N$ be a finitely generated submodule of $M$. If $S^{-1}N \subseteq  S^{-1}K$ for a submodule
$K$ of $M$, then there exists an $s \in S$ such that $sN \subseteq K$.
\end{lem}
\begin{proof}
This is straightforward.
\end{proof}

Let $S$ be a  multiplicatively closed subset of $R$. Recall that the saturation $S^*$ of $S$ is defined as $S^*=\{x \in R : x/1\  is \ a\ unit  \ of\ S^{-1}R \}$. It is obvious that $S^*$ is a  multiplicatively closed subset of $R$ containing $S$ \cite{Gr92}.

\begin{prop}\label{p2.2}
Let $S$ be a  multiplicatively closed subset of $R$ and $M$ be an R-module. Then we have the following.
\begin{itemize}
\item [(a)] If $N$ is a secondary submodule of $M$ such that $S \cap \sqrt{Ann_R(N)}=\emptyset$, then $N$ is an $S$-secondary submodule of $M$. In fact if $S \subseteq u(R)$ and $N$ is an $S$-secondary submodule of $M$, then $N$ is a secondary submodule of $M$.
\item [(b)] If $S_1 \subseteq S_2$ are  multiplicatively closed subsets of $R$ and $N$ is an $S_1$-secondary submodule of $M$, then $N$ is an $S_2$-secondary submodule of $M$ in case $\sqrt{Ann_R(N)} \cap S_2=\emptyset$.
\item [(c)] $N$ is an $S$-secondary submodule of $M$ if and only if $N$ is an $S^*$-secondary submodule of $M$
\item [(d)] If $N$ is a finitely generated $S$-secondary submodule of $M$, then $S^{-1}N$ is a secondary submodule of $S^{-1}M$
\end{itemize}
\end{prop}
\begin{proof}
(a) and (b) These are clear.

(c) Assume that $N$ is an $S$-secondary submodule of $M$.  We claim that $\sqrt{Ann_R(N)} \cap  S^*=\emptyset$. To see this assume that there exists a fixed $x \in \sqrt{Ann_R(N)} \cap  S^*$  As $x \in S^*$,  $x/1$ is a unit of $S^{-1}R$ and so $(x/1)(a/s)=1$ for some $a \in R$ and $s \in S$. This yields that $us = uxa$ for some $u \in S$. Now we have that $us = uxa \in  \sqrt{Ann_R(N)} \cap  S$,  a
contradiction. Thus, $\sqrt{Ann_R(N)} \cap  S^*=\emptyset$. Now as $S\subseteq S^*$, by part (b), $N$ is an $S^*$-secondary submodule of $M$. Conversely, assume that $N$ is an $S^*$-secondary submodule of $M$. Let $rN \subseteq K$. As $N$ is an $S^*$-secondary submodule of $M$, there is a fixed $x \in  S^*$ such that $xr \in \sqrt{Ann_R(N)}$ or $xN \subseteq K$. As $x/1$ is a unit of $S^{-1}R$, there exist $u, s \in S$ and $a \in R$ such that $us = uxa$. Then
note that $(us)r = uaxr \in \sqrt{Ann_R(N)}$ or $(us)N=ua(xN) \subseteq K$. Therefore, $N$ is an $S$-secondary submodule of $M$.

(d) If $S^{-1}N=0$, then as $N$ is finitely generated, there is an $s \in S$ such that $s \in \sqrt{Ann_R(N)}$ by Lemma \ref{l2.1}. This implies that $\sqrt{Ann_R(N)} \cap  S\not=\emptyset$, a contradiction. Thus  $S^{-1}N\not=0$. Now let $r/t \in S^{-1}R$. As $N$ is an $S$-secondary submodule of $M$, there is a fixed $s \in S$ such that $rsN=sN$ or $(rs)^tN=0$ for some $t \in \Bbb N$. If $rsN=sN$, then $(r/s)S^{-1}N=S^{-1}N$. If  $(rs)^tN=0$, then $(r/s)^tS^{-1}N=0$, as needed.
\end{proof}

\begin{cor}\label{c2.2}
Let $M$ be an R-module and set $S=\{1\}$. Then  every secondary submodule of $M$ is an $S$-secondary submodule of $M$.
\end{cor}
\begin{proof}
Let $N$ be a secondary submodule of $M$. Then as $N \not=0$, we have $1 \not  \in \sqrt{Ann_R(N)}$. Hence $S \cap \sqrt{Ann_R(N)}=\emptyset$. Thus the result follows from Proposition \ref{p2.2} (a).
\end{proof}

The following examples show that the converses of Proposition \ref{p2.2} (a) and (d) are not true in general.
\begin{ex}\label{e2.2}
Take the $\Bbb Z$-module $M=\Bbb Z_{p^\infty}\oplus \Bbb Z_2$ for a prime number $p$. Then  for each $n \in \Bbb N$, $2^n(\Bbb Z_{p^\infty}\oplus \Bbb Z_2)=\Bbb Z_{p^\infty}\oplus 0$ implies that $M$ is not a secondary $\Bbb Z$-module.  Now,
take the  multiplicatively closed subset $S =\Bbb Z\setminus \{0\}$ and put $s=2$. Then $2rM=\Bbb Z_{p^\infty}\oplus 0=2M$ for all $r \in \Bbb Z$ and so  $M$ is an $S$-secondary $\Bbb Z$-module.
\end{ex}

\begin{ex}\label{e22.2}
Consider the $\Bbb Z$-module $M=\Bbb Q \oplus \Bbb Q$, where $\Bbb Q$ is the field of rational numbers. Take the submodule
$N = \Bbb Z \oplus 0$ and the  multiplicatively closed subset  $S =\Bbb Z\setminus \{0\}$. Then one can see that $N$ is not an $S$-secondary submodule of $M$. Since $S^{-1}\Bbb Z =\Bbb Q$
is a field, $S^{-1}(\Bbb Q \oplus \Bbb Q)$ is a vector space so that a non-zero submodule $S^{-1}N$ is a secondary submodule of $S^{-1}(\Bbb Q \oplus \Bbb Q)$.
\end{ex}

\begin{prop}\label{p2.4}
Let $M$ be an $R$-module and $S$ be a  multiplicatively closed subset of $R$.  If $N$ is an $S$-secondary submodule of $M$, then $\sqrt{Ann_R(N)}$ is an $S$-prime ideal of $R$.
\end{prop}
\begin{proof}
Let $ab \in \sqrt{Ann_R(N)}$ for some $a, b \in R$. Then $(ab)^nN=0$ for some $n \in \Bbb N$. As $N$ is an $S$-secondary submodule of $M$, there exists a fixed $s \in S$ such that $asN=sN$ or $(as)^tN=0$ for some $t \in \Bbb N$ and $bsN=sN$ or $(bs)^hN=0$  for some $h \in \Bbb N$.
If $(as)^tN=0$ or $(bs)^hN=0$ we are done. If $asN=sN$ and $bsN=sN$, then $0=(bas)^nN=sN$. Thus $s \in \sqrt{Ann_R(N)}$, a contradiction. Thus in any case, $(as)^tN=0$ or $(bs)^hN=0$, as needed.
\end{proof}

The following example shows that the converse of Proposition \ref{p2.4} is not
true in general.
\begin{ex}\label{e2.4}
Consider $M = \Bbb Z\oplus Q$ as a $\Bbb Z$-module. Take the  multiplicatively closed subset  $S =\Bbb Z\setminus \{0\}$. Then $Ann_R(M) = 0$ is an $S$-prime ideal of $R$. But $M$ is not an $S$-secondary submodule of $M$.
\end{ex}

Let $R_i$ be a commutative ring with identity, $M_i$ be an $R_i$-module for each $i = 1, 2,..., n$, and $n \in \Bbb N$. Assume that
$M = M_1\times M_2\times ...\times M_n$ and $R = R_1\times R_2\times ...\times R_n$. Then $M$ is clearly
an $R$-module with componentwise addition and scalar multiplication. Also,
if $S_i$ is a multiplicatively closed subset of $R_i$ for each $i = 1, 2,...,n$,  then
$S = S_1\times S_2\times ...\times S_n$ is a multiplicatively closed subset of $R$. Furthermore,
each submodule $N$ of $M$ is of the form $N = N_1\times N_2\times...\times N_n$, where $N_i$ is a
submodule of $M_i$. 
\begin{thm}\label{t2.6}
Let  $M = M_1 \times M_2$ be an $R = R_1 \times R_2$-module and $S = S_1\times S_2$ be a   multiplicatively closed subset of
$R$, where $M_i$ is an $R_i$-module and $S_i$ is a   multiplicatively closed subset of $R_i$ for each $i = 1, 2$. Let $N=N_1\times N_2$ be a submodule of $M$. Then the following are equivalent:
\begin{itemize}
\item [(a)] $N$ is an $S$-secondary submodule of $M$;
\item [(b)] $N_1$ is an $S_1$-secondary submodule of $M_1$ and $\sqrt{Ann_{R_2}(N_2)} \cap S_2\not=\emptyset$ or $N_2$ is an $S_2$-secondary submodule of $M_2$ and $\sqrt{Ann_{R_1}(N_1)} \cap S_1\not=\emptyset$.
\end{itemize}
\end{thm}
\begin{proof}
$(a)\Rightarrow (b)$
Let $N = N_1 \times N_2$ be an $S$-secondary submodule of $M$. Then $\sqrt{Ann_R(N)}= \sqrt{Ann_{R_1}(N_1)}\times \sqrt{Ann_{R_2}(N_2)}$ is an $S$-prime ideal of $R$ by Proposition \ref{p2.4}. By \cite[Lemma 2.13]{satk19}, either $\sqrt{Ann_R(N_1)} \cap S_1\not = \emptyset$ or  $\sqrt{Ann_R(N_2)} \cap S_2\not = \emptyset$. We may assume that $\sqrt{Ann_R(N_1)} \cap S_1\not = \emptyset$. We show that
$N_2$ is an $S_2$-secondary submodule of $M_2$. To see this, let $r_2N_2 \subseteq K_2$ for some $r_2 \in R_2$ and a submodule $K_2$ of $M_2$. Then $(1,r_2)(N_1 \times N_2) \subseteq M_1 \times K_2$. As $N$ is an $S$-secondary submodule of $M$, there exists a fixed $(s_1,s_2) \in S$ such that $(s_1,s_2)(N_1 \times N_2) \subseteq M_1 \times K_2$ or $((s_1,s_2)(1,r_2))^t(N_1 \times N_2)=0$ for some $t \in \Bbb N$. It follows that  $s_2N_2 \subseteq K_2$ or $(s_2r_2)^tN_2=0$ and so $N_2$ is an $S_2$-secondary submodule of $M_2$. Similarly, if $\sqrt{Ann_{R_2}(N_2)} \cap S_2\not=\emptyset$, one can see that $N_1$ is an $S_1$-secondary submodule of $M_1$.

$(b)\Rightarrow (a)$
Assume that $N_1$ is an $S_1$-secondary submodule of $M_1$ and $\sqrt{Ann_{R_2}(N_2)} \cap S_2\not=\emptyset$. Then there exists a fixed $s_2 \in \sqrt{Ann_{R_2}(N_2)} \cap S_2$.  Let $(r_1, r_2)(N_1 \times N_2) \subseteq K_1 \times K_2$ for some $r_i \in R_i$ and submodule $K_i$ of $M_i$, where $i = 1, 2$. Then $r_1N_1 \subseteq K_1$.  As $N_1$ is an $S_1$-secondary submodule of $M_1$,  there exists a fixed $s_1 \in S_1$ such that $s_1N_1 \subseteq K_1$ or $(s_1r_1)^tN_1 =0$ for some $t \in \Bbb N$.
Now we set $s =(s_1, s_2)$. Then $s(N_1 \times N_2) \subseteq K_1 \times K_2$ or $(s(r_1, r_2))^t(N_1 \times N_2)=0$.
Therefore, $N$ is an $S$-secondary submodule of $M$. Similarly one can show that if  $N_2$ is an $S_2$-secondary submodule of $M_2$ and $\sqrt{Ann_{R_1}(N_1)} \cap S_1\not=\emptyset$, then  $N$ is an $S$-secondary submodule of $M$.
\end{proof}

\begin{thm}\label{t2.7}
Let  $M = M_1 \times M_2\times ... \times M_n$ be an $R = R_1 \times R_2\times ...\times R_n$-module and $S = S_1\times S_2\times ... \times S_n$ be a   multiplicatively closed subset of
$R$, where $M_i$ is an $R_i$-module and $S_i$ is a  multiplicatively closed subset of $R_i$ for each $i = 1, 2,...,n$. Let $N=N_1\times N_2\times ... \times N_n$ be a submodule of $M$. Then the following are equivalent:
\begin{itemize}
\item [(a)] $N$ is an $S$-secondary submodule of $M$;
\item [(b)] $N_i$ is an $S_i$-secondary submodule of $M_i$ for some $i \in \{1,2,...,n\}$ and $\sqrt{Ann_{R_j}(N_j)} \cap S_j\not=\emptyset$ for all $j \in \{1,2,...,n\}-\{ i \}$.
\end{itemize}
\end{thm}
\begin{proof}
We apply induction on $n$. For $n = 1$, the result is true. If $n = 2$, then the result follows from Theorem \ref{t2.6}. Now assume that parts (a) and (b) are equal when $k < n$. We shall prove $(b)\Leftrightarrow (a)$ when $k = n$. Let
$N=N_1\times N_2\times ... \times N_n$.  Put $\acute{N}=N_1\times N_2\times ... \times N_{n-1}$ and $\acute{S} = S_1\times S_2\times ... \times S_{n-1}$. Then by Theorem \ref{t2.6}, the necessary and sufficient condition for $N$ is an $S$-secondary submodule of $M$ is that $\acute{N}$ is an $\acute{S}$-secondary submodule of $\acute{M}$ and  $\sqrt{Ann_{R_n}(N_n)} \cap S_n\not=\emptyset$ or
 $N_n$ is an $S_n$-secondary submodule of $M_n$ and $\sqrt{Ann_{\acute{R}}(\acute{N})} \cap \acute{S}\not=\emptyset$, where  $\acute{R}=R_1\times R_2\times ... \times R_{n-1}$. Now the result follows from the induction hypothesis.
\end{proof}

\begin{lem}\label{t22.8}
Let $S$ be a  multiplicatively closed subset of $R$ and $N$ be an $S$-secondary submodule of an $R$-module $M$. Then the following statements hold for some $s \in S$.
\begin{itemize}
\item [(a)] $sN \subseteq \acute{s}N$ for all $\acute{s}\in S$.
\item [(b)] $(\sqrt{Ann_R(N)}:_R\acute{s}) \subseteq (\sqrt{Ann_R(N)}:_Rs)$ for all $\acute{s} \in S$.
\end{itemize}
\end{lem}
\begin{proof}
(a) Let $N$ be an $S$-secondary submodule of $M$. Then there is an $s \in S$ such that $rN \subseteq K$ for each $r \in R$ and a submodule $K$ of $M$ implies that $sN \subseteq K$ or $(sr)^tN=0$ for some $t \in \Bbb N$. Let $L$ be a completely irreducible submodule of $M$ such that $\acute{s}N \subseteq L$. Then $sN \subseteq L$ or $(\acute{s}s)^tN=0$. As $\sqrt{Ann_R(N)} \cap S =\emptyset$, we get that $sN \subseteq L$. Thus $sN \subseteq \acute{s}N$ by Remark \ref{r2.1}.

(b) This follows from Proposition \ref{p2.4} and \cite[Lemma 2.16 (ii)]{satk19}.
\end{proof}

\begin{prop}\label{t2.8}
Let $S$ be a  multiplicatively closed subset of $R$ and $N$ be a finitely generated submodule of $M$ such that $\sqrt{Ann_R(N)} \cap S=\emptyset$. Then the following are equivalent:
\begin{itemize}
\item [(a)] $N$ is an $S$-secondary submodule of $M$;
\item [(b)] $S^{-1}N$ is a secondary submodule of $S^{-1}M$ and there is an $s \in S$ satisfying $sN \subseteq \acute{s}N$ for all $\acute{s}\in S$.
\end{itemize}
\end{prop}
\begin{proof}
$(a)\Rightarrow (b)$
This follows from Proposition \ref{p2.2} (d) and  Lemma \ref{t22.8}.

$(b)\Rightarrow (a)$
Let $aN \subseteq K$ for some $a \in R$ and a submodule $K$ of $M$. Then $(a/1)(S^{-1}N) \subseteq S^{-1}K$. Thus by part (b),
$S^{-1}N\subseteq S^{-1}K$ or $(a/1)^t(S^{-1}N)=0$ for some $t \in \Bbb N$. Hence by Lemma \ref{l2.1}, $s_1N \subseteq K$ or $(s_2a)^tN=0$ for some $s_1,s_2 \in S$. By part (b), there is an $s \in S$ such that $sN \subseteq (s_1)^tN$ and $sN \subseteq (s_2)^tN \subseteq (0:_Ma^t)$. Therefore, $sN \subseteq K$ or $(as)^tN \subseteq sa^tN=0$, as desired.
\end{proof}

\begin{thm}\label{t2.8}
Let $S$ be a  multiplicatively closed subset of $R$ and $N$ be a submodule of an $R$-module $M$ such that $\sqrt{Ann_R(N)} \cap S=\emptyset$. Then $N$ is an $S$-secondary submodule of $M$ if and only if $sN$ is a secondary submodule of $M$ for some $s \in S$.
\end{thm}
\begin{proof}
Let $sN$ be a secondary submodule of $M$ for some $s \in S$. Let $aN \subseteq K$ for some $a \in R$ and a submodule $K$ of $M$.   As
$asN \subseteq K$ and $sN$ is a secondary submodule of $M$, we get that $sN \subseteq K$ or $a^tsN=0$ and so $(as)^tN=0$ for some $t \in \Bbb N$, as needed. Conversely, assume that $N$ is an $S$-secondary submodule of $M$. Then there is an $s\in  S$ such that if $aN \subseteq K$ for some $a \in R$ and a submodule $K$ of $M$, then $sN \subseteq K$ or $(sa)^tN=0$  for some $t \in \Bbb N$. Now we show that $sN$ is a
secondary submodule of $M$. Let $a \in R$. As $asN \subseteq asN$, by assumption, $sN \subseteq asN$ or $(as^2)^tN=0$  for some $t \in \Bbb N$. If $sN \subseteq asN$, then there is nothing to show. Assume that $sN \not\subseteq asN$. Then $(as^2)^tN=0$ and so $a \in (\sqrt{Ann_R(N)}:_Rs^2) \subseteq (\sqrt{Ann_R(N)}:_Rs)$ by Lemma \ref{t22.8} (b). Thus, we can conclude that $(as)^tN=0$  for some $t \in \Bbb N$, as desired.
\end{proof}

The set of all maximal ideals of $R$ is denoted by $Max(R)$.
\begin{thm}\label{t2.9}
Let $S$ be a  multiplicatively closed subset of $R$ and $N$ be a submodule of an $R$-module $M$ such that $\sqrt{Ann_R(N)} \subseteq Jac(R)$, where $Jac(R)$ is the Jacobson radical of $R$. Then the following statements are equivalent:
\begin{itemize}
\item [(a)] $N$ is a secondary submodule of $M$;
\item [(b)] $\sqrt{Ann_R(N)}$ is a prime ideal of $R$ and $N$ is an  $(R\setminus \mathfrak{M})$-secondary submodule of $M$ for each $\mathfrak{M} \in Max(R)$.
\end{itemize}
\end{thm}
\begin{proof}
$(a)\Rightarrow (b)$ Let  $N$ be a secondary submodule of $M$. Since $\sqrt{Ann_R(N)} \subseteq Jac(R)$,  $\sqrt{Ann_R(N)} \subseteq \mathfrak{M}$  for each $\mathfrak{M} \in Max(R)$ and so $\sqrt{Ann_R(N)} \cap (R\setminus \mathfrak{M})=\emptyset$. Now the result follows from Proposition \ref{p2.2} (a).

$(b)\Rightarrow (a)$
Let $\sqrt{Ann_R(N)}$ be a prime ideal of $R$ and $N$ be an  $(R\setminus \mathfrak{M})$-secondary submodule of $M$ for each $\mathfrak{M} \in Max(R)$. Let $a \in R$ and $a \not \in \sqrt{Ann_R(N)}$. We show that $aN=N$. Let $\mathfrak{M} \in Max(R)$. Then as $aN \subseteq aN$,  there exists a fixed $s_{\mathfrak{M}} \in R \setminus \mathfrak{M}$ such that
$s_{\mathfrak{M}} N \subseteq aN$ or $(s_{\mathfrak{M}}a)^tN=0$ for some $t \in \Bbb N$. As $\sqrt{Ann_R(N)}$ is a prime ideal of $R$ and $s_{\mathfrak{M}} \not \in \sqrt{Ann_R(N)}$,  we have $as_{\mathfrak{M}} \not \in \sqrt{Ann_R(N)}$ and so $s_{\mathfrak{M}} N \subseteq aN$. Now consider the set
$$
\Omega =\{s_{\mathfrak{M}} : \exists \: \mathfrak{M} \in Max(R), s_{\mathfrak{M}}  \not \in \mathfrak{M} \: and \: s_{\mathfrak{M}}N \subseteq aN\}.
$$
Then we claim that $\Omega= R$. To see this, take any maximal ideal $\acute{\mathfrak{M}}$  containing
$\Omega$. Then the definition of $\Omega$ requires that there exists a fixed $s_{\acute{\mathfrak{M}}} \in \Omega$ and $s_{\acute{\mathfrak{M}}} \not \in \acute{\mathfrak{M}}$. As $\Omega \subseteq \acute{\mathfrak{M}}$, we have
$s_{\acute{\mathfrak{M}}}  \in \Omega \subseteq \acute{\mathfrak{M}}$, a contradiction. Thus, $\Omega= R$ and this yields
$$
 1 = r_1s_{\mathfrak{M_1}} +  r_2s_{\mathfrak{M_2}} +... + r_ns_{\mathfrak{M_n}}
$$
for some $r_i \in R$ and $s_{\mathfrak{M_i}} \in R \setminus \mathfrak{M_i}$ with $s_{\mathfrak{M_i}}N \subseteq aN$, where $\mathfrak{M_i} \in Max(R)$ for each $i = 1, 2, ...,n$. This yields that
$$
 N = (r_1s_{\mathfrak{M_1}} +  r_2s_{\mathfrak{M_2}} +... + r_ns_{\mathfrak{M_n}})N \subseteq aN.
$$
Therefore, $N \subseteq aN$ as needed.
\end{proof}

Now we determine all secondary submodules of a module over a quasilocal ring in terms of $S$-secondary submodules.

\begin{cor}\label{cc2.9}
Let $S$ be a  multiplicatively closed subset of  a quasilocal ring $(R, \mathfrak{M})$ and $N$ be a submodule of an $R$-module $M$. Then the following statements are equivalent:
\begin{itemize}
\item [(a)] $N$ is a secondary submodule of $M$;
\item [(b)] $\sqrt{Ann_R(N)}$ is a prime ideal of $R$ and $N$ is an  $(R\setminus \mathfrak{M})$-secondary submodule of $M$.
\end{itemize}
\end{cor}
\begin{proof}
This follows from Theorem \ref{t2.9}.
\end{proof}

\begin{prop}\label{t2.16}
Let $S$ be a  multiplicatively closed subset of $R$ and $f : M \rightarrow \acute{M}$ be a monomorphism of R-modules. Then we have the following.
\begin{itemize}
  \item [(a)] If $N$ is an $S$-secondary submodule of $M$, then $f(N)$ is an $S$-secondary submodule of $\acute{M}$.
  \item [(b)] If $\acute{N}$ is an $S$-secondary submodule of $\acute{M}$ and $\acute{N} \subseteq f(M)$, then $f^{-1}(\acute{N})$ is an $S$-secondary submodule of $M$.
 \end{itemize}
\end{prop}
\begin{proof}
(a) Since  $\sqrt{Ann_R(N)} \cap S=\emptyset$ and $f$ is a monomorphism of R-modules, we have $\sqrt{Ann_R(f(N))} \cap S=\emptyset$. Let $r \in R$. Since $N$ is an $S$-secondary submodule of $M$, there exists a fixed $s \in S$ such that $srN=sN$ or $(sr)^tN=0$ for some $t \in \Bbb N$. Thus
 $srf(N)=sf(N)$ or $(sr)^tf(N)=0$, as needed.

(b)  $\sqrt{Ann_R(\acute{N})} \cap S=\emptyset$ implies that $\sqrt{Ann_R(f^{-1}(\acute{N}))} \cap S=\emptyset$.
Now let $r \in R$. As $\acute{N}$ is an $S$-secondary submodule of $\acute{M}$, there exists a fixed $s \in S$ such that $sr\acute{N}=s\acute{N}$ or $(sr)^t\acute{N}=0$ for some $t \in \Bbb N$.
Therefore $srf^{-1}(\acute{N})=sf^{-1}(\acute{N})$ or $(sr)^tf^{-1}(\acute{N})=0$, as requested.
\end{proof}

An $R$-module $M$ is said to be a \emph{comultiplication module} if for every submodule $N$ of $M$ there exists an ideal $I$ of $R$ such that $N=(0:_MI)$, equivalently, for each submodule $N$ of $M$, we have $N=(0:_MAnn_R(N))$ \cite{AF07}.
\begin{prop}\label{p2.17}
Let $S$ be a  multiplicatively closed subset of $R$, $M$ a comultiplication $R$-module, and  let $N$ be an $S$-secondary submodule of $M$. Suppose that $N \subseteq K+H$ for some
submodules $K, H$ of $M$. Then $s(0:_M\sqrt{Ann_R(N)})\subseteq K$ or $s(0:_M\sqrt{Ann_R(N)}) \subseteq H$ for some $s \in S$.
\end{prop}
\begin{proof}
As $N \subseteq K+H$, we have $Ann_R(K)Ann_R(H) \subseteq \sqrt{Ann_R(N)}$. This implies that there exists a fixed $s \in S$  such that
 $sAnn_R(K)\subseteq \sqrt{Ann_R(N)}$ or  $sAnn_R(H) \subseteq \sqrt{Ann_R(N)}$ since by Proposition \ref{p2.4}, $\sqrt{Ann_R(N)}$ is an $S$-prime ideal of $R$. Therefore, $(0:_M\sqrt{Ann_R(N)})\subseteq ((0:_MAnn_R(K)):_Ms)$ or $(0:_M\sqrt{Ann_R(N)})\subseteq ((0:_MAnn_R(H)):_Ms)$. Now as $M$ is a comultiplication $R$-module, we have  $s(0:_M\sqrt{Ann_R(N)}) \subseteq K$ or $s(0:_M\sqrt{Ann_R(N)}) \subseteq H$ as needed.
\end{proof}

Let $M$ be an $R$-module. The idealization $R(+)M =\{(a,m): a \in R, m \in  M\}$ of $M$ is
a commutative ring whose addition is componentwise and whose multiplication is defined as $(a,m)(b,\acute{m}) =
(ab, a\acute{m} + bm)$ for each $a, b \in R$, $m, \acute{m}\in M$ \cite{Na62}. If $S$ is a  multiplicatively closed subset of $R$ and $N$ is a submodule of $M$, then $S(+)N = \{(s, n): s \in S, n \in N\}$ is a  multiplicatively closed subset of $R(+)M$ \cite{DD02}.

\begin{prop}\label{p2.18}
Let $M$ be an $R$-module and let $I$ be an ideal of $R$ such that $I \subseteq Ann_R(M)$.  Then the following are equivalent:
\begin{itemize}
\item [(a)] $I$ is a secondary ideal of $R$;
\item [(b)] $I(+)0$ is a secondary ideal of $R(+)M$.
\end{itemize}
\end{prop}
\begin{proof}
This is straightforward.
\end{proof}

\begin{thm}\label{t2.19}
Let $S$ be a  multiplicatively closed subset of $R$, $M$ be an $R$-module, and $I$ be an ideal of $R$ such that $I \subseteq Ann_R(M)$ and $I \cap S=\emptyset$. Then the following are equivalent:
\begin{itemize}
\item [(a)] $I$ is an $S$-secondary ideal of $R$;
\item [(b)] $I(+)0$ is an $S(+)0$-secondary ideal of $R(+)M$;
\item [(c)] $I(+)0$ is an $S(+)M$-secondary ideal of $R(+)M$.
\end{itemize}
\end{thm}
\begin{proof}
$(a)\Rightarrow (b)$
Let $(r, m ) \in R(+)M$. As $I$ is an $S$-secondary ideal of $R$, there exists a fixed $s \in S$ such that $rsI=sI$ or $(rs)^tI=0$  for some $t \in \Bbb N$.
If  $(rs)^tI=0$, then $((r,m)(s,0))^t(I(+)0)=0$. If $rsI=sI$, then we claim that $(r,m)(s,0)(I(+)0)=(s,0)(I(+)0)$. To see this let
$(sa,0)=(s,0)(a,0) \in (s,0)(I(+)0)$. As $rsI=sI$, we have $sa=rsb$ for some $b \in I$. Thus as $b \in I \subseteq Ann_R(M)$,
$$
(sa,0)=(srb,0)=(srb,smb)=(sr,sm)(b,0)=(s,0)(r,m)(b,0).
$$
Hence,
$(s,0)(a,0) \in (r,m)(s,0)(I(+)0)$ and so $(s,0)(I(+)0) \subseteq (r,m)(s,0)(I(+)0)$.
Since the inverse inclusion is clear we reach the claim.

$(b)\Rightarrow (c)$
Since $S(+)0 \subseteq S(+)M$, the result follows from Proposition \ref{p2.2} (b).

$(c)\Rightarrow (a)$
Let $r \in R$. As $I(+)0$ is an $S(+)M$-secondary ideal of $R(+)M$, there exists a fixed $(s, m) \in S(+)M$ such that $(r,0)(s,m)(I(+)0)=(s,m)(I(+)0)$ or $((r,0)(s,m))^t(I(+)0)=0$ for some $t \in \Bbb N$.
If $((r,0)(s,m))^t(I(+)0)=0$, then
$$
0=((r,0)(s,m))^t(a,0)=(rs,rm)^t(a,0)=((rs)^ta,(rm)^ta)=((rs)^ta,0)
$$
far each $a \in I$. Thus $(rs)^tI=0$. If
$(r,0)(s,m)(I(+)0)=(s,m)(I(+)0)$, then we claim that $rsI=sI$. To see this let
$sa\in sI$. Then for some $b \in I$, we have
$$
(sa,0)=(sa,am)=(s,m)(a,0)=(s,m)(r,0)(b,0)=(srb,rmb)=(srb,0).
$$
Hence,
$sa \in rsI$ and so $sI \subseteq srI$, as needed.
\end{proof}

Let $P$ be a prime ideal of $R$ and $N$ be a
submodule of an $R$-module $M$. \emph{The $P$-interior of $N$ relative to $M$}
is defined (see \cite[2.7]{AF11}) as the set
$$
I^M_P(N)= \cap \{L \mid  L \\\ is \\\ a \\\ completely \\\
irreducible
\\\ submodule \\\ of \\\ M\\\ and
$$
$$
 rN\subseteq L \\\ for \\\ some \\\ r \in R-P \}.
$$
Let $R$ be an integral domain. A submodule $N$ of
an $R$-module $M$ is said to be a \emph{cotorsion-free submodule of $M$} (the dual of torsion-free) if $I^M_0(N)=N$ and is a \emph{cotorsion submodule of $M$} (the dual of torsion) if $I^M_0(N)=0$. Also, $M$ said to be \emph{cotorsion} (resp. \emph{cotorsion-free}) if $M$ is a cotorsion (resp. cotorsion-free) submodule of $M$ \cite{AF12}.

\begin{defn}\label{d2.3}
Let $M$ be an $R$-module and let $S$ be a  multiplicatively closed subset of $R$ such that $\sqrt{Ann_R(M)} \cap S=\emptyset$.  We say that $M$ is a \emph{quasi $S$-cotorsion-free module} in the case that there is an $s\in S$ such that if $rM \subseteq L$, where $r \in R$ and $L$ is a completely irreducible submodule of $M$, then $sM \subseteq L$ or $(rs)^t=0$ for some $t \in \Bbb N$.
\end{defn}

\begin{prop}\label{t2.3}
Let $M$ be an $R$-module and $S$ be a  multiplicatively closed subset of $R$.
Then the following statements are equivalent:
\begin{itemize}
  \item [(a)] $M$ is an $S$-secondary $R$-module;
  \item [(b)] $P=\sqrt{Ann_R(M)}$ is an $S$-prime ideal of $R$ and the $R/P$-module $M$ is
a quasi $S$-cotorsion-free module.
\end{itemize}
\end{prop}
\begin{proof}
$(a) \Rightarrow (b)$. By Proposition \ref{p2.4} (a), $\sqrt{Ann_R(M)}$ is an $S$-prime
ideal of $R$. Now let $L$ be a completely irreducible submodule
of $M$ and $r \in R$ such that $rM \subseteq L$.
Then there exists a fixed $s \in S$ such that $sM \subseteq L$ or $(rs)^tM=0$ for some $t \in \Bbb N$ because $M$ is $S$-secondary.
Thus $sM \subseteq L$ or $rs \in P=0_{R/P}$ as required.

$(b) \Rightarrow (a)$. As $\sqrt{Ann_R(M)}$ is an $S$-prime ideal of $R$, $\sqrt{Ann_R(M)} \cap S=\emptyset$. Suppose that there exist $r \in R$ and completely irreducible submodule $L$ of $M$ such that $rM\subseteq L$.
By assumption, there is an $s \in S$ such that $sM \subseteq L$ or $(rs)^t=0_{R/P}$ for some $t \in \Bbb N$. Thus  $sM \subseteq L$ or $rs \in P=\sqrt{Ann_R(M)}$ as needed.
\end{proof}

\begin{thm}\label{t2.21}
Let $M$ be a module over an integral domain $R$. Then the following are equivalent:
\begin{itemize}
\item [(a)]
 $M$ is a cotorsion-free $R$-module;
\item [(b)]
$M$ is a quasi $(R \setminus P)$-cotorsion-free for each prime ideal $P$ of $R$;
\item [(c)]
$M$ is a quasi $(R \setminus \mathfrak{M})$-cotorsion-free for each maximal ideal $\mathfrak{M}$ of $R$.
\end{itemize}
\end{thm}
\begin{proof}
$(a)\Rightarrow (b)$
This is clear.

$(b)\Rightarrow (c)$
This is obvious.

$(c)\Rightarrow (a)$
Let  $M$ be quasi $(R \setminus \mathfrak{M})$-cotorsion-free for each maximal ideal $\mathfrak{M}$ of $R$. Let $aM \subseteq L$ for some $a \in  R$ and a completely irreducible submodule $L$ of $M$. Assume that $a \not=0$.  Take  a maximal ideal $\mathfrak{M}$ of $R$. As $M$ is quasi $(R \setminus \mathfrak{M})$-cotorsion-free, there exists a fixed $s_\mathfrak{M} \in R \setminus \mathfrak{M}$ such that
$s_\mathfrak{M}M \subseteq L$ or $(as_\mathfrak{M})^t=0$ for some $t \in \Bbb N$.  As $R$ is an integral domain, $(as_\mathfrak{M})^t \not=0$ and so $s_\mathfrak{M}M \subseteq L$. Now set
$$
\Omega =\{s_{\mathfrak{M}} : \exists \: \mathfrak{M} \in Max(R), s_{\mathfrak{M}}  \not \in \mathfrak{M} \: and \: s_{\mathfrak{M}}M \subseteq L\}.
$$
A similar argument as in the proof of Theorem \ref{t2.9} shows that $\Omega =R$. Thus we have $\langle s_{\mathfrak{M_1}}\rangle + \langle s_{\mathfrak{M_2}}\rangle + ... + \langle s_{\mathfrak{M_n}}\rangle = R$ for some  $s_{\mathfrak{M_i}} \in \Omega$. This implies that
 $M=(\langle s_{\mathfrak{M_1}}\rangle + \langle s_{\mathfrak{M_2}}\rangle + ... + \langle s_{\mathfrak{M_n}}\rangle)M \subseteq L$ and hence $M=L$. This means that $M$ is a cotorsion-free $R$-module.
\end{proof}

Let $M$ be an $R$-module. The dual notion of $Z_R(M)$, the
set of zero divisors of $M$, is denoted by $W_R(M)$ and defined by
$$
W(M)= \{ a \in R: aM \not =M \}.
$$

\begin{thm}\label{t2.10}
Let $S$ be a  multiplicatively closed subset of $R$ and $M$ be a finitely generated comultiplication $R$-module with $\sqrt{Ann_R(M)} \cap S=\emptyset$. If each non-zero submodule of $M$ is $S$-secondary, then $W_R(M)=\sqrt{Ann_R(M)}$.
\end{thm}
\begin{proof}
Assume that every non-zero submodule of $M$ is an $S$-secondary submodule of $M$.
 Let $a \in W_R(M)$. Then $aM \not=M$. Since $M$ is $S$-secondary, there exists a fixed $s\in S$ such that $saM=sM$ or $(sa)^tM=0$ for some $t \in \Bbb N$. If $saM=sM$, then $s \in (saM:_RM)$. Now put $N=(0:_M(saM:_RM))$ and note that $s \in S \cap Ann_R(N) \not=\emptyset$. Thus, $N$ is not $S$-secondary and so by part (a), we have $N=(0:_M(saM:_RM))=0$. Now as $M$ is a comultiplication $R$-module, one can see that $M(saM:_RM)=M$. By  \cite[Corollary 2.5]{AM69}, $1 - x \in Ann_R(M) \subseteq (saM:_RM)$ for some $x \in (saM:_RM)$ since $M$ is a finitely generated $R$-module. This implies that $(saM:_RM)=R$  and so $saM=M$. It follows that $aM=M$, which is a contradiction. Therefore, $(sa)^tM=0$. Then $s \in \sqrt{Ann_R(a^tM)}$ and so  $S\cap \sqrt{Ann_R(a^tM)}\not = \emptyset$. Hence by assumption, $a^tM=0$. Thus, we get $W_R(M) = \sqrt{Ann_R(M)}$.
\end{proof}

\begin{defn}\label{d2.10}
Let $S$ be a  multiplicatively closed subset of $R$ and $P$ be a submodule of an $R$-module $M$ with $\sqrt{(P :_R M)} \cap S=\emptyset$. Then we say that $P$ is an \textit{$S$-primary submodule} if there exists a fixed $s\in S$ and whenever $am \in P$, then either $sa \in \sqrt{(P :_R M)}$ or $sm \in P$  for each $a \in R$ and $m \in M$.
\end{defn}

An $R$-module $M$ is said to be a
\emph{multiplication module} if for every submodule $N$ of $M$
there exists an ideal $I$ of $R$ such that $N=IM$ \cite{Ba81}.

\begin{thm}\label{t32.10}
Let $S$ be a  multiplicatively closed subset of $R$ and $M$ be a finitely generated multiplication $R$-module with $\sqrt{Ann_R(M)} \cap S=\emptyset$. If each proper submodule of $M$ is $S$-primary, then $Z_R(M)=\sqrt{Ann_R(M)}$.
\end{thm}
\begin{proof}
Let $a \in Z_R(M)$. Then there is a $0 \not= \acute{m} \in M$ with $a\acute{m} = 0$. Since the zero submodule is $S$-primary
and $a\acute{m} = 0$, there is a fixed $s \in S$ so that $sa \in Ann_R(M)$ or $s\acute{m} = 0$. If $s\acute{m} = 0$; we have $s\in Ann_R(\acute{m})$. Now put $\acute{P} = Ann_R(\acute{m})M$ and note that $S \cap \sqrt{(\acute{P} :_R M)}=\emptyset$. Thus, we have $\acute{P} = Ann_R(\acute{m})M = M$. By \cite[Corollary 2.5]{AM69}, $1 - x \in Ann_R(M) \subseteq Ann_R(\acute{m})$ for some $x \in Ann_R(\acute{m})$. This yields that $Ann_R(\acute{m}) = R$ and so
$\acute{m} = 0$; which is a contradiction. We have $sa \in Ann_R(M)$. Then we can conclude that $s \in ((0:_Ma):_R M)$
and hence by assumption $ (0:_Ma)= M$.  Thus, we get $a \in Ann_R(M)$. Therefore, $Z_R(M) =\sqrt{Ann_R(M)}$.
\end{proof}

\begin{cor}\label{c2.10}
Let $S$ be a  multiplicatively closed subset of $R$. If $M$ is a finitely generated multiplication and comultiplication $R$-module with $\sqrt{Ann_R(M)} \cap S=\emptyset$, then the following statements are equivalent:
\begin{itemize}
\item [(a)] Each non-zero submodule of $M$ is $S$-secondary;
\item [(b)]  $Z_R(M)=W_R(M) = \sqrt{Ann_R(M)}$;
\item [(c)] Each proper submodule of $M$ is an $S$-primary submodule of $M$.
\end{itemize}
\end{cor}
\begin{proof}
As $M$ is a multiplication and comultiplication $R$-module, one can see that  $Z_R(M)=W_R(M)$. Now the results follows from Theorem \ref{t2.10} and  Theorem \ref{t32.10}.
\end{proof}

\begin{ex}\label{e2.10}
For any prime integer $p$ and any positive integer $n\geq 2$, consider the $\Bbb Z$-module $\Bbb Z_{p^n}$. Take $S=\Bbb Z -p\Bbb Z$. We know that $\Bbb Z_{p^n}$ is a finitely generated multiplication and comultiplication $\Bbb Z$-module. Then by Corollary \ref{c2.10}, the $\Bbb Z$-module $\Bbb Z_{p^n}$ has a non-zero submodule which is not $S$-secondary and a proper submodule which is not $S$-primary.
\end{ex}
\bibliographystyle{amsplain}

\begin{thebibliography}{10}
\bibitem{DD02}
D.D. Anderson and M. Winders,  \emph{Idealization of a module}, Journal of Commutative Algebra \textbf{1} (1) 2009, 3-56.

\bibitem{AF07}
H.~Ansari-Toroghy and F.~Farshadifar, \emph{The dual notion of multiplication modules}, Taiwanese J. Math. \textbf{11} (4) (2007), 1189--1201.

\bibitem{AF11}
H.~Ansari-Toroghy and F.~Farshadifar, \emph{On the dual notion of prime submodules}, Algebra Colloq. \textbf{19} (Spec 1) (2012), 1109-1116.

\bibitem{AF101}
H. Ansari-Toroghy and F. Farshadifar, \emph{The dual notion of some generalizations of prime submodules}, Comm. Algebra, \textbf{39} (2011), 2396-2416.

\bibitem{AF12}
H.~Ansari-Toroghy and F.~Farshadifar, \emph{On the dual notion of prime submodules (II)}, Mediterr. J. Math., \textbf{9} (2) (2012), 329-338.

\bibitem{AF09}
H.~Ansari-Toroghy and F.~Farshadifar, \emph{Strong comultiplication modules}, CMU. J. Nat. Sci. \textbf{8} (1) (2009), 105--113.

\bibitem{AM69}
M.F. Atiyah and I.G. Macdonald, \emph{Introduction to commutative algebra}, Addison-Wesley, 1969.

\bibitem{Ba81}
A. Barnard, \emph{Multiplication modules}, J. Algebra, \textbf{71} (1981), 174--178.

\bibitem{Da78}
J.~Dauns, \emph{Prime submodules}, J. Reine Angew. Math. \textbf{298} (1978),
156-181.

\bibitem{Fa95}
 C. Faith, \emph{Rings whose modules have maximal submodules},
Publ. Mat. {\bf {39}} (1995), 201-214.

\bibitem{FF22}
F.~Farshadifar, \emph{$S$-second submodules of a module}, Algebra and Discrete Mathematics, to appear.

\bibitem{FHo06}
L.~Fuchs, W.~Heinzer, and B.~Olberding, \emph{Commutative ideal theory without
finiteness conditions: Irreducibility in the quotient filed}, in : Abelian Groups, Rings, Modules, and Homological Algebra, Lect. Notes Pure Appl. Math.  \textbf{249}, 121--145, (2006).

\bibitem{Gr92}
R. Gilmer,  \emph{Multiplicative Ideal Theory}, Queen’s Papers in Pure and Applied Mathematics, No. 90. Kingston, Canada:
Queen’s University, 1992.

\bibitem{Na62}
M. Nagata, \emph{Local Rings}, New York, NY, USA: Interscience, 1962.

\bibitem{satk19}
E.S. Sevim, T. Arabaci, $\ddot{U}$. Tekir, and S. Koc, \emph{On S-prime submodules}, Turkish Journal of Mathematics, \textbf{43} (2) (2019), 1036-1046.

\bibitem{WK16}
F. Wang  and H. Kim, \emph{Foundations of Commutative Rings and Their Modules}, Singapore: Springer, 2016.

\bibitem{Y01}
S.~Yassemi, \emph{The dual notion of prime submodules}, Arch. Math. (Brno)
\textbf{37} (2001), 273--278.

\end{thebibliography}

\end{document}